\documentclass[11pt]{amsart}
\usepackage[margin=3cm]{geometry}
\usepackage{bm}
\usepackage{amsthm,amsmath,mathtools}
\usepackage{amssymb}
\usepackage{mathrsfs}
\usepackage{hyperref,xcolor}
\usepackage[foot]{amsaddr}
\usepackage{stmaryrd}
\usepackage{graphicx}

\newcommand{\R}{\mathbb{R}}

\newcommand{\Z}{\mathbb{Z}}
\newcommand{\e}{\epsilon}

\newcommand{\Ric}{\mathrm{Ric}}

\DeclareMathOperator{\ind}{ind}
\DeclareMathOperator{\nul}{nul}

\DeclareMathOperator{\supp}{supp}
\DeclareMathOperator{\area}{area}

\newtheorem*{thm*}{Theorem}
\newtheorem*{namedthm}{\namedthmname}
\newtheorem{lem}{Lemma}
\newtheorem*{lem*}{Lemma}
\newtheorem{prop}{Proposition}

\newtheorem*{coro*}{Corollary}

\theoremstyle{definition}

\newtheorem{rmk}{Remark}
\newtheorem*{que}{Question}

\newcounter{namedthm}
\makeatletter
\newenvironment{named}[1]
  {\def\namedthmname{#1}%
   \refstepcounter{namedthm}%
   \namedthm\def\@currentlabel{#1}}
  {\endnamedthm}
\makeatother

\title{Allen-Cahn equation and degenerate minimal hypersurface}
\author{Jingwen Chen\textsuperscript{1}, Pedro Gaspar\textsuperscript{2}}
\address{\parbox{\linewidth}{
\textsuperscript{1} Department of Mathematics, University of Pennsylvania, \\
David Rittenhouse Lab,
209 South 33rd Street,
Philadelphia, PA 19104 \\ \textsuperscript{2} Facultad de Matem\'aticas, Pontificia Universidad Cat\'olica de Chile \\ Avenida Vicuña Mackenna 4860, Santiago, Chile \smallskip}}
\email{jingwch@sas.upenn.edu, pedro.gaspar@mat.uc.cl}

\begin{document}

\begin{abstract}
In this short note, we present new observations and examples concerning the existence and rigidity of solutions to the Allen-Cahn equation with degenerate minimal hypersurfaces as their limit interfaces.
\end{abstract}

\maketitle

\section{Introduction}

In this article, we present new observations and examples related to degenerate minimal hypersurfaces and solutions of the Allen-Cahn equation approximating those minimal hypersurfaces in compact Riemannian manifolds. In particular, we describe a Riemannian metric of nonnegative Ricci curvature on $S^n$, $n \geq 3$, in which:
\begin{enumerate}
    \item[(i)] there exist \emph{connected}, degenerate, stable minimal hypersurfaces which cannot be obtained as the limit interface of solutions to the Allen-Cahn equation;
    \item[(ii)] the space of minimal hypersurfaces obtained from min-max solutions of the Allen-Cahn equation is strictly smaller than those produced by Almgren-Pitts min-max theory; and
    \item[(iii)] a family of nondegenerate index $1$ solutions of the Allen-Cahn equation which has a degenerate stable limit interface.
\end{enumerate}
We observe that the nonexistence of solutions only depends on the nonnegativity of the Ricci curvature of the ambient manifold, and we do not require any condition on the multiplicity of the limit interface (while multiplicity one is often assumed in min-max theory \cite{ChodoshMantoulidis, MarquesNevesMultiplicity,Zhou}).\bigskip

Let $(M^n,g)$ be a compact Riemannian manifold. The \emph{Allen-Cahn equation}:
\begin{equation} \label{AC}
\Delta_g u -  {\textstyle \frac{1}{\e^2}}W'(u)=0 \tag{AC}
\end{equation}
on $M$ serves as a model for phase transition and separation phenomena \cite{AC}. Here, $u$ is a real-valued function defined on $M$, $\epsilon$ is a positive parameter, and $W$ represents a symmetric, nonnegative double-well potential with wells located at $\pm 1$ -- a typical example is $W(u) = (1-u^2)^2/4$. Solutions to the Allen-Cahn equation are closely related to minimal hypersurfaces, roughly speaking, the nodal set of a sequence of solutions $u_{\e_j}$ of \eqref{AC} with $\e_j \searrow 0$ accumulates on a minimal hypersurface as $j \to \infty$, and we call this minimal hypersurface the \emph{limit interface} obtained from $u_{\e_j}$. The Allen-Cahn equation has been extensively studied and widely used in recent decades in view of their connection to minimal surfaces; we refer the reader to \cite{ChenGaspar,Guaracominmax} and the references therein for some applications. A central topic in the study of the Allen-Cahn equation is the existence of solutions with a given minimal hypersurface as their limit interface. Pacard and Ritor{\'e} \cite{pacard2003constant} positively answered this topic in the nondegenerate case, with relatively few works addressing the degenerate case. We now briefly describe some convergence and existence results in these directions.\medskip

Hereafter we denote by $\sigma = \int_{-1}^1 \sqrt{W(t)/2} dt$ the energy of the \emph{heteroclinic solution} $\mathbb{H}_\e(t)$ of \eqref{AC} on $\R$, namely, the unique bounded solution in $\R$ (modulo translation) such that $\mathbb{H}_\e(t) \to \pm 1$ when $t \to \pm \infty$. For $n \geq 3$, the convergence of solutions of \eqref{AC} to limit interfaces can be described more precisely as follows. Let $\{u_j\}$ be a sequence of solutions of \eqref{AC} with $\e=\e_j \searrow 0$. Suppose that their energy $E_{\e_j}(u_j)$ and Morse index $\ind_{\e_j}(u_j)$ are bounded. It follows from the combined works of J. Hutchinson, Y. Tonegawa, N. Wickramasekera and M. Guaraco \cite{HT,TW,Guaracominmax} that we can find a subsequence of $u_j$ such that their associated varifolds converge to a stationary $(n-1)$ varifold $V$ on $M$, and $\frac{1}{\sigma} V$ is integral. Moreover, $u_j$ converges uniformly to $\pm 1$ in compact subsets of $M \setminus \supp \|V\|$, and $\supp \|V\|$ is a smooth, embedded, minimal hypersurface in $M$ away from a closed set of Hausdorff dimension $\leq (n-8)$, which is called a \emph{limit interface} obtained from $u_j$. We also note that a strong convergence result was proved, more recently, for generic metrics on $3$-dimensional manifolds by O. Chodosh and C. Mantoulidis in \cite{ChodoshMantoulidis}.\smallskip

Conversely, Pacard and Ritor{\'e} \cite{pacard2003constant,Pacard} used an infinite dimensional \emph{Lyapunov-Schmidt reduction} method to prove that any \emph{nondegenerate} minimal hypersurface $\Gamma$ in $(M^n,g)$ can be obtained as a limit interface of solutions $u_{\epsilon}$ of \eqref{AC} whose nodal set is a small perturbation of $\Gamma$, provided $\Gamma$ separates $M$ in the following sense: there is a smooth function $f_{\Gamma}$ on $M$ such that $0$ is regular value for $f_{\Gamma}$ and $\Gamma = f_{\Gamma}^{-1}(0)$. For the case of geodesics in Euclidean domains this was proved by Kowalczyk in \cite{Kowalczyk}, and a similar construction for complete minimal surfaces of finite total curvature was carried out by del Pino-Kowalczyk-Wei in \cite{dPKW}. The existence of solutions whose nodal sets have multiple components that converge to a minimal hypersurface with higher multiplicity was addressed by del Pino-Kowalczyk-Wei-Yang in \cite{dPKWY}, under the nondegeneracy assumption and assuming a positivity condition on $|A_\Gamma|^2+\Ric_g(\vec{n_\Gamma},\vec{n_\Gamma})$. Later, Caju and the second named author \cite{caju2019solutions} observed that the construction of \cite{Pacard} can be carried out for any minimal hypersurface $\Gamma$ such that all the Jacobi fields of $\Gamma$ are generated by ambient isometries. We also mention the construction of Allen-Cahn solutions whose nodal set approximate the \emph{Simons minimal cone} in $\R^N$, for $N\geq 9$, by del Pino-Kowalczyk-Wei \cite{dPKW-DeGiorgi}, and which are connected to a well-known conjecture of De Giorgi about entire, monotone solutions of \eqref{AC} in $\R^N$. Finally, we note that variational constructions of solutions of \eqref{AC} that approximate a given nondegenerate minimal hypersurface have been recently obtained by G. De Philippis and A. Pigati in \cite{DePhilippisPigati}, and alternatively by E.M. Silva in \cite{Silva}.

Nevertheless, there are relatively few works that discuss obstructions to the existence and the convergence of solutions of \eqref{AC} to general, possibly degenerate minimal hypersurfaces. In \cite{guaraco2019multiplicity}, Guaraco-Marques-Neves discuss a connected, non-separating strictly stable minimal surface that cannot be obtained as a limit interface of solutions of the Allen-Cahn equation with multiplicity $>1$. This obstruction arises from their improved convergence estimates for strictly stable limit interfaces, which implies that such minimal hypersurfaces can only arise as multiplicity one interfaces. In \cite{mantoulidis2022double}, Mantoulidis described an interesting family of examples of degenerate, disconnected minimal hypersurfaces in a product manifold of the type $M \times S^1$ that cannot be obtained as limit interfaces of solutions of the Allen-Cahn equation. These are unions of an even number of slices $M \times \{\theta_i\}$ for certain $\theta_i \in S^1$ so that these surfaces separate the ambient manifold $M \times S^1$. Note that these examples are degenerate stable, as any rotation with respect to the $S^1$ factor gives rise to a positive Jacobi field.\medskip

The first main result of this paper is:

\begin{named}{Theorem A} \label{thma}
Let $\tilde g$ be a Riemannian metric on $S^n$ with nonnegative Ricci curvature that is invariant by the action of $\mathrm{O}(n)$ in the last $n$ coordinates in $S^n\subset \R^{n+1}$. Suppose that:
\begin{enumerate}
    \item[(i)] The reflection $(x_1, x_2\cdots,x_{n+1}) \mapsto (-x_1,x_2\cdots,x_{n+1})$ is an isometry of $(S^n,\tilde g)$.
    \item[(ii)] There exists $\delta \in (0,1)$ and $r>0$ such that the region $\{x \in S^n \mid |x_1|\leq\delta\}$ is isometric to a finite cylinder $[-r,r] \times S^{n-1}$, so that, for any $t \in (-\delta,\delta)$, the spheres $\Sigma_t=\{x_1=t\}$ are degenerate stable minimal (in fact, totally geodesic) hypersurfaces in $(S^n,\tilde g)$.
\end{enumerate}
Then the middle sphere $\Sigma_0$ is the only one among the family of minimal spheres $\{\Sigma_t\}_{t \in (-\delta,\delta)}$ which can be obtained as the limit interface of a family $\{u_\e\}$ of solutions of the Allen-Cahn equation on $(S^n,\tilde g)$ (with any positive integer multiplicity).
\end{named}

We refer to Section \ref{frankel_rotational} for the precise statement. Additionally, we will construct examples of metrics $\tilde g$ with these properties as warped product metrics on an interval $[0,R]\times S^{n-1}$ that degenerate at $\{0\} \times S^{n-1}$ and $\{R\} \times S^{n-1}$ to give rise to rotationally symmetric metrics on $S^{n}$ which can be intuitively pictured as a finite cylinder capped with nonnegatively curved disks. The minimal hypersurfaces $\Sigma_t$ for $t \in (-\delta, 0) \cup (0, \delta)$ provide examples of connected, degenerate, minimal hypersurfaces which cannot be given as limit interfaces associated to solutions of the Allen-Cahn equation. See Remark \ref{nonsymmetric} for a possible generalization to a class of metrics with nonnegative Ricci curvature without reflexive or rotational symmetry. \smallskip

In summary, we present a rough idea of why solutions to the Allen-Cahn equation \eqref{AC} are more rigid compared to minimal hypersurfaces. Minimal hypersurfaces only rely on local geometry near the hypersurfaces, whereas solutions to \eqref{AC} are affected by the global geometry of the ambient manifold. To further elaborate, in the example we constructed in \ref{thma}, locally the minimal hypersurfaces $\{t\} \times S^n$, where $t \in (-\delta, \delta)$, are essentially identical modulo translations. However, the nonnegativity of the Ricci curvature ensures that there exists at most one minimal sphere $\{t\} \times S^n$ that can be approximated by solutions to \eqref{AC}, and the reflexive symmetry indicates that the middle sphere $\{0\} \times S^n$ is the only one. In fact, by property (i) of \ref{thma}, one can construct $\mathrm{O}(n)$-invariant solutions $v_\e$ of the Allen-Cahn equation on $(S^n,\tilde g)$, for sufficiently small $\e>0$, which vanish precisely along the middle sphere $\Sigma_0$ by using the results of Brezis-Oswald \cite{BO} and by an odd reflection -- a similar construction can be found in \cite{caju2022ground}.

We mention here that the construction of Neumann Allen-Cahn solutions that approximate \emph{degenerate} line segments in bounded planar domains which contain a rectangle was carried out by Kowalczyk \cite{KowalczykThesis} and Alikakos-Kowalczyk-Fusco \cite{AFK}, with a control on the location of the corresponding nodal set (under some curvature assumptions on the boundary of the domain near the corners of the rectangles). Intuitively, the metrics studied here can be thought as the closed analogue of these domains -- see also Remark \ref{nonsymmetric} below.

\begin{rmk}
The results of \cite{mantoulidis2022double} show that the global isometry requirement in \cite{caju2019solutions} cannot be relaxed to a \emph{local} integrability condition. The examples given by our \ref{thma} show that this remains true even if one assumes that the minimal hypersurface is connected.\medskip
\end{rmk}

One of the central ingredients in the construction is the \emph{Frankel-type} property for solutions of the Allen-Cahn equation on compact manifolds with nonnegative Ricci curvature, which was first proved by F. Hiesmayr \cite[Proposition 6]{H} for compact manifolds of positive Ricci curvature.

\begin{named}{Theorem (Frankel Property)}\label{nonnegative Ricci FP}
Let $(M^n,g)$ be a compact Riemannian manifold with $\Ric_g \geq 0$. If $u_1$ and $u_2$ are distinct solutions of \eqref{AC} with $u_1 \leq u_2$, then one of the solutions is constant. 
\end{named}

In this regard, \ref{thma} is related to the strong rigidity of nonintersecting minimal hypersurfaces in compact Riemannian manifolds of nonnegative Ricci curvature proved by P. Petersen and F. Wilhelm in \cite{PetersenWilhelm}.\medskip

Next, we study the min-max width and the least area minimal hypersurface of $(S^n, \tilde g)$. It follows from the maximum principle and the monotonicity formula that the spheres $\Sigma_t$, for $t \in (-\delta, \delta)$, represent the minimal hypersurface of $(S^n, \tilde g)$ with the least area. On the other hand, the family of spheres $\{\Sigma_t\}_{t \in [-1,1]}$ provides a sweepout of $(S^n,\tilde g)$ in either the sense of Simon-Smith or Almgren-Pitts, see \cite{ColdingDeLellis,MarquesNevesSurvey}. This sweepout has any of the minimal spheres $\Sigma_t$, for $t \in [-\delta,\delta]$, as the maximum for the area functional, and thus provides an optimal sweepout with respect to the width associated to its homotopy class (see Proposition \ref{width}). In light of the comparison between the phase-transitions and the Almgren-Pitts min-max approaches for the area functionals recently proved by Dey in \cite{Dey}, we will show:

\begin{named}{Theorem B} \label{thmb}
There exists a dimensional constant $R=R(n)>0$ with the following property. Let $\tilde g$ be a Riemannian metric on $S^n$ as in Theorem \ref{thma}. If $r>R(n)$, then there exists a $\Z_2$-homotopy class $\Pi$ of continuous maps $X \to \mathcal{Z}_{n-1}(S^n,\tilde g;\mathcal{F};\Z_2)$ into the space of $(n-1)$-cycles mod $2$ such that
    \[\mathbf{L}(\Pi) = \omega_1(S^n,\tilde g) \quad \text{and} \quad \{\Sigma_0\} = \mathbf{C}_{AC}(\tilde \Pi) \subsetneq \mathbf{C}_{AP}(\Pi),\]
where $X$ is a cubical complex with nontrivial double cover $\tilde X \to X$, and $\tilde \Pi$ denote the space of all $\Z_2$-equivariant maps $\tilde X \to W^{1,2}(S^n)\setminus\{0\}$.
\end{named}

We refer to section \ref{min-max} for notation. As a consequence of \ref{thmb}, we get:

\begin{named}{Corollary} \label{corollary}
Let $\tilde g$ be a Riemannian metric on $S^n$ as in \ref{thmb}. There exists $\e_0>0$ such that, for all $\e \in (0,\e_0)$, the symmetric solution $v_\e$ is the only least energy nonconstant solution of the Allen-Cahn equation on $(S^n,g)$. 
\end{named}

\begin{rmk} \label{index}
Let $\{u_j\}$ be a sequence of nonconstant solutions of \eqref{AC} with $\e=\e_j \searrow 0$ in a compact (oriented) Riemannian manifold $(M,g)$ with $\Ric_g \geq 0$ which has as its associated limit interface the multiplicity one embedded, two sided, minimal hypersurface $\Sigma \subset M$.  

If $\Sigma$ is stable, then it follows from the second variation formula that any constant function is in the kernel of its Jacobi operator, and hence $\Sigma$ has nullity 1. On the one hand, by the rigidity results of \cite{farina2013stable} for stable solutions of semilinear equations in manifolds of nonnegative Ricci curvature, we see that the solutions $u_\e$ cannot be stable, and hence $\ind(u_\e) \geq 1$. On the other hand, by the upper semicontinuity of the spectrum of the linearized Allen-Cahn equation \cite{ChodoshMantoulidis},
        \[\ind(u_\e)+\nul(u_\e) \leq \ind(\Sigma) + \nul(\Sigma) = 1.\]
Therefore, $u_\e$ must be a nondegenerate index 1 solution. This conclusion can be applied to the symmetric solutions $v_\e$ in $(S^n,\tilde g)$, providing a degenerate example where the second-order convergence of the Allen-Cahn energy functional to the area functional -- that is, the codimension one volume -- does not hold.
\end{rmk}

As previously noted in \cite{mantoulidis2022double}, these rigidity phenomena are compatible with the stronger convergence results in the context of phase transitions, such as \cite{ChodoshMantoulidis,CMSurfaces}, in comparison with the solution of the Multiplicity One Conjecture by X. Zhou \cite{Zhou} for the Almgren-Pitts min-max, and by Wang-Zhou \cite{wangzhouexistence} for the Simon-Smith min-max -- and also the related results by Calabi-Cao \cite{Calabicao} and Aiex \cite{Aiex} in dimension 2. \bigskip

Finally, we mention some related questions and examples. Motivated by the geometric picture, we expect that any constant $r > 0$ should work in \ref{thma} (ii), and it will be true if the following question is answered positively, as indicated by Remark \ref{comparehalfsphere}.

\begin{que}

Let $f: S^n \to [1,\infty)$ be a smooth function, and consider the smooth manifold $S^n_f = \{f(x) x \mid x \in S^n\} \subset \R^{n+1}$, endowed with the metric induced by $\R^{n+1}$. 

\begin{enumerate}
\item[(i)] Is the $1$-width of $S^n_f$ minimized (among such smooth functions) if and only if $f \equiv 1$?

\item[(ii)] Among those functions $f$ for which $S^n_f$ has nonnegative Ricci curvature, is the area of the least area minimal hypersurface in $S^n_f$ minimized if and only if $f \equiv 1$?
\end{enumerate}

\end{que}

\begin{figure}[ht]
\centering
\includegraphics[width=.2\textwidth]{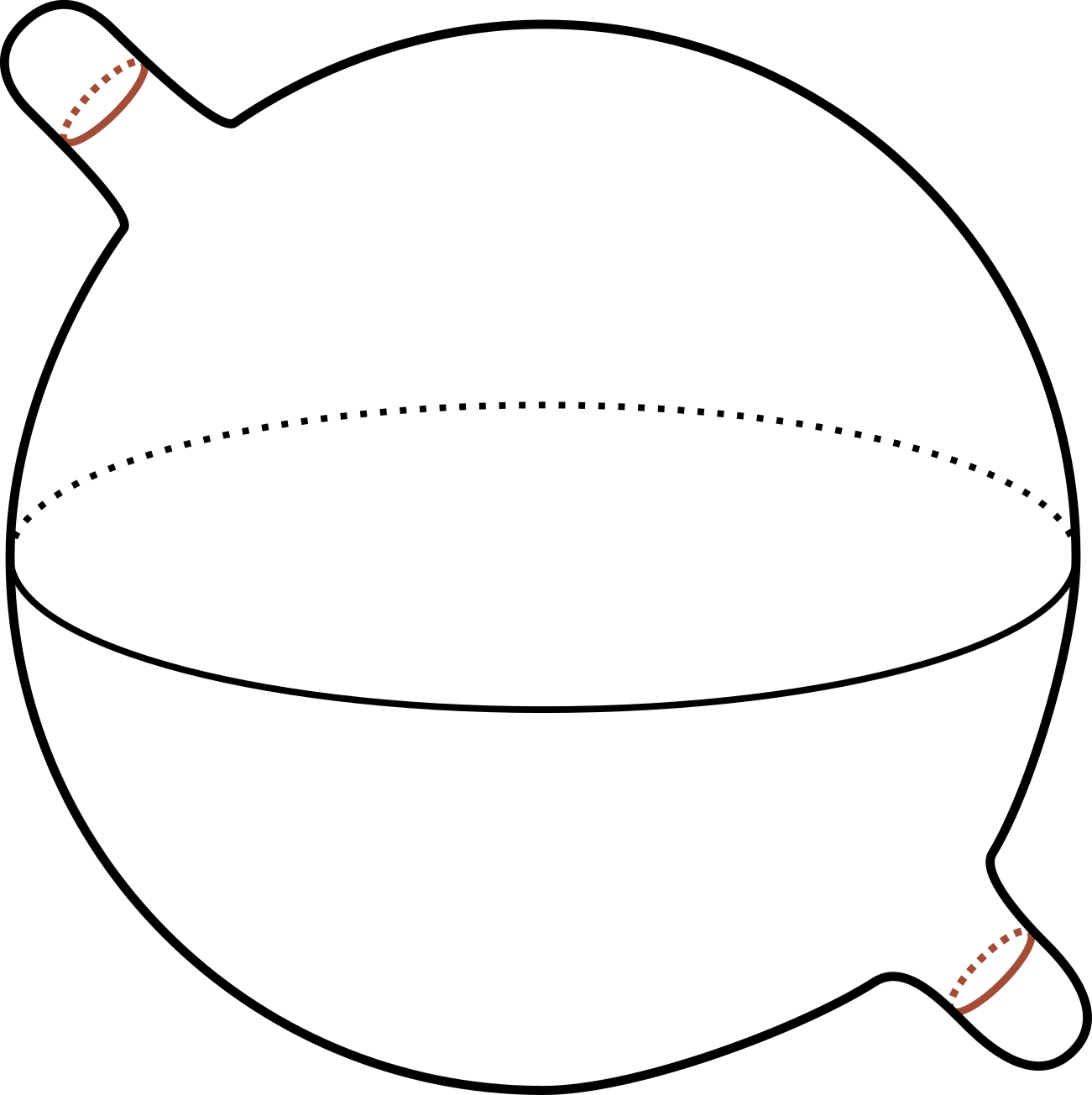}
\caption{Spherical graph with low area minimal surfaces}
\label{graph}
\end{figure}

Roughly speaking, the above question asks that if we have a larger manifold containing $S^n$ that is a graph over $S^n$, does the area of the least area minimal hypersurface and $1-$width also increase? The curvature condition in part $(ii)$ is used to rule out the possibility of low area stable minimal surfaces, see the picture below. We also note that by the characterization of least area minimal surfaces by L. Mazet and H. Rosenberg \cite{mazetrosenberg} and the work of R. Souam \cite{Souam}, if $n=3$ then either the area of this surface equals the $1$-width of the ambient manifold, or it is a stable surface which is homeomorphic to a torus or a sphere and the Ricci curvature vanishes along the normal bundle of the surface. We also mention the related inequalities for the normalized width of conformally flat $3$-spheres by Ambrozio-Montezuma \cite{AmbrozioMontezuma}.\medskip

It is an interesting problem to study conditions under which we can obtain the uniqueness or rigidity of the Allen-Cahn approximation of a given minimal hypersurface. Guaraco-Marques-Neves \cite{guaraco2019multiplicity} proved the Allen-Cahn approximation of a separating, nondegenerate minimal hypersurface is unique. This is proved by showing that the multiplicity-one convergence implies strong (graphical, $C^{2,\alpha}$) estimates, which means that these solutions eventually (for sufficiently small $\e$) belong to the function space used by Pacard \cite{Pacard} (see also Pacard-Ritor\'e \cite{pacard2003constant}) to employ fixed point methods, and hence are unique. Hiesmayr \cite{H} proved the rigidity for equatorial spheres certain Clifford hypersurfaces in the round $n$-sphere, up to isometry. 

To be precise, the rigidity problems consist in showing that there exists at most one solution to \eqref{AC} (up to isometry) with a given minimal hypersurface as its limit interface. Geometrically, this problem can be compared to the question: does the set of areas of all (embedded) minimal hypersurface have accumulation points? We also mention that other accumulation phenomena for minimal hypersurfaces were recently studied by Song-Zhou \cite{song2021generic}.

In contrast, we present an example related to the main construction of the article that shows that a \emph{degenerate} minimal hypersurface may have two distinct Allen-Cahn approximations (modulo isometry). Once again, this example indicates that the solutions to the Allen-Cahn equation on a manifold depend on the global geometry.\bigskip

\noindent \textbf{Example.} Consider the warped product metric
    \[g = dt^2 + f(t)^2 \cdot g_{{round}} \]
on the cylinder $M=\R \times S^n$, where $f$ is a positive $C^2$ function defined by $f(0)=0$ and $$f(t) = t^{6}\cdot\left(\sin\left(\frac{2\pi}{t}\right)+1\right)+1,$$  and $g_{{round}}$ is the round metric on $S^n$. Since $$f'(t)= t^4 \left(6t\left(\sin\left(\frac{2\pi}{t}\right) + 1\right) - 2\pi \cos \left(\frac{2\pi}{t}\right) \right),$$
there exists an increasing sequence of negative critical points $t^-_i \nearrow 0$, and a decreasing sequence of positive critical points $t^+_i \searrow 0$. By the computations of \cite[p. 206, Proposition 35]{ONeill}, the spheres $\{t^{\pm}_i\} \times S^n$ are totally geodesic and hence minimal. Moreover, when $i$ is sufficiently large, we have $\sin\left(\frac{2\pi}{t^{\pm}_i}\right) \simeq 1$, therefore $f''(t^{\pm}_i)\neq 0$ and the minimal hypersurface $\{t^{\pm}_i\} \times S^n$ is nondegenerate. Thus $\{t^{\pm}_i\} \times S^n$ can be approximated by solutions of \eqref{AC} by the result of Pacard and Ritor\'e \cite{pacard2003constant}.

Now we study the Allen-Cahn approximation of the hypersurface $\{0\} \times S^n$. By considering variations by parallel spheres $\{t\} \times S^n$ of radius $f(t)$, we see that $\{0\} \times S^n$ is a degenerate stable minimal hypersurface since it supports a positive Jacobi field. If we consider the Allen-Cahn approximation of the nondegenerate minimal hypersurfaces $\{t^{\pm}_i\} \times S^n$, by a diagonal sequence argument, we get two Allen-Cahn approximations, one is extracted from $\{t^{+}_i\} \times S^n$, and another is extracted from $\{t^{-}_i\} \times S^n$. From the choice of our function $f$, there is no reflexive symmetry about $0$, hence these two Allen-Cahn approximations are distinct.

\subsection*{Acknowledgements}

PG was supported by ANID (Agencia Nacional de Investigaci\'on y Desarrollo, Chile) FONDECYT Iniciaci\'on grant.

\section{The Frankel property for the Allen-Cahn equation and rotational metrics} \label{frankel_rotational}

In this section, we construct Riemannian metrics on $S^n$ for which there exists a connected, separating, minimal hypersurface that cannot be obtained as the limit interfaces of solutions of the Allen-Cahn equation (with any positive integer multiplicity). We start with a few useful results about solutions of this equation.

\subsection{The Frankel-type property}

We say that a Riemannian manifold $(M^n,g)$ satisfies the \emph{Frankel property} if any two closed, two-sided, smooth, embedded minimal hypersurfaces of $(M^n,g)$ intersect each other. This holds, for instance, whenever $\Ric_g>0$, and it is a useful tool in the study of minimal surfaces, see e.g. \cite{MarquesNevesRicci}. We also note that if $\Ric_g \geq 0$, then the existence of nonintersecting minimal surfaces imposes a strong rigidity phenomenon, as proved by Petersen-Wilhelm \cite{PetersenWilhelm}, see also Choe-Fraser \cite{ChoeFraser}

In \cite{H}, F. Hiesmayr proved an analog of the Frankel-type property of solutions of the Allen-Cahn equation in compact manifolds of positive Ricci curvature. We will need to extend his result to the nonnegative Ricci curvature setting, as stated in the introduction. We follow the same proof as in \cite[Proposition 6]{H} coupled with the classification of stable solutions for certain semilinear elliptic PDEs under nonnegative Ricci curvature assumption, as demonstrated by A. Farina, Y. Sire, and E. Valdinoci \cite{farina2013stable}. We include the complete proof here for completeness. \medskip

\begin{proof}[Proof of \ref{nonnegative Ricci FP}]
By Serrin's maximum principle (see \cite[Proposition 7.2]{guaraco2019multiplicity}), we must have $u_1 < u_2$. By \cite[Theorems 1 and 2]{farina2013stable}, a stable solution of \eqref{AC} in a compact Riemannian manifold with nonnegative Ricci curvature is constant. Therefore, if we assume $u_1$ is not constant, then it is unstable. This means that we can find a positive function $\varphi$ such that $\Delta \varphi-\e^{-2}W''(u_1)\varphi +\lambda_1 \varphi =0$ for some $\lambda_1<0$, and hence $u_1+\theta\varphi$ is a subsolution for sufficiently small $\theta \in (0,1)$ (depending on $\varphi$). In fact, we can argue as in \cite{guaraco2019multiplicity} to see that $|W'(t)-W'(s)-W''(s)(t-s)| \leq C (t-s)^2$ for some $C>0$ depending on $W$ only, and hence 
    \[|W'(u_1+\theta \varphi) - W'(u_1) - W''(u_1)\theta \varphi|\leq C(u_1+\theta\varphi - u_1)^2 = C\varphi^2 \theta^2 \]
and
    \begin{align*}
        &\Delta(u_1+\theta\varphi)-\e^{-2} W'(u_1+\theta\varphi) \\
        &\qquad = \Delta u_1 -\theta \lambda_1 \varphi - \e^{-2}\left(W'(u_1 + \theta\varphi) - W''(u_1)\theta\varphi\right) \\
        &\qquad \geq  \Delta u_1 -\theta \lambda_1 \varphi - \e^{-2} W'(u_1) -C\e^{-2}\varphi^2\theta^2= \theta\varphi(- \lambda_1 -C\e^{-2}\varphi\theta )
    \end{align*}
which is nonnegative provided $\theta < \e^2(-\lambda_1)/(C\|\varphi\|_{\infty})$. Additionally, we can choose $\theta$ sufficiently small so that $u_1 + \theta \varphi < u_2$ everywhere on $M$. By the maximum principle for parabolic equations, the solution $u(\cdot,t)$ of the parabolic Allen-Cahn equation with initial data $u(\cdot,0) = u_1+\theta \varphi$ satisfies $u(\cdot,t)<u_2$ everywhere and $M$ and for all $t\geq 0$. Moreover, it is strictly increasing with respect to $t$, since the time derivative $\partial_t u$ solves the linearized parabolic equation and 
    \[\partial_t u(\cdot, 0) = \e \Delta(u_1+\theta\varphi) - \e^{-1}W'(u_1+\theta\varphi) \geq 0,\]
and this implies the strict inequality $\partial_t u(\cdot, t)>0$ using the maximum principle once again. Thus (see e.g. \cite[Lemma 2.3]{GasparGuaraco}), it subconverges as $t \to +\infty$ to a solution $u_+$ of \eqref{AC} with 
    \[u_1 < u_1+\theta \varphi< u_+ \leq u_2.\] 
As noted in \cite{H}, the solution $u_+$ must be stable. Otherwise, we can argue as we did for  $u_1+\theta \varphi$ to show that if $\psi$ is the first positive eigenfunction for the linearized operator $L_{\e,u_+}$ and if $\delta>0$ is sufficiently small, then $u_+-\delta \psi$ is a supersolution of \eqref{AC}. By choosing $\delta$ so that $u_1+\theta\varphi<u_+-\delta\psi$, we reach a contradiction, since the solution $v$ of the parabolic Allen-Cahn equation having $u_+-\delta\psi$ as its initial datum is strictly decreasing and hence acts as a barrier to prevent $u(\cdot,t) \to u_+$ (subsequentially), for we have 
    \[u_1 + \theta\varphi = u(\cdot,0) < u(\cdot,t) < v(\cdot,t) < v(\cdot,1) <u_+ -\delta \psi < u_+, \quad \text{for all} \ t\geq 1.\]
Therefore $u_+$ is stable and, using the nonnegativity of Ricci curvature again, it must be constant. Using $u_1 <u_+ \leq u_2 \leq 1$ and that $u_1$ cannot have a constant sign, we conclude that  $u_2$ is the constant solution $+1$.
\end{proof}

We will also use the following comparison result, proved by Guaraco-Marques-Neves, which is a consequence of Serrin's maximum principle \cite{HanLin}:

\begin{lem}(\cite[Corollary 7.4]{guaraco2019multiplicity}) \label{comparison}
Suppose $u,v$ are solutions to \eqref{AC} with parameter $\e > 0$ on an open bounded domain $\Omega \subset M$. If $v$ is continuous and positive on $\overline\Omega$ and $u = 0$ on $\partial \Omega$, then $v > u$ on $\overline \Omega$.
\end{lem}

\subsection{Construction of the metric}

First, we construct a suitable symmetric Riemannian metric on $S^n$ which contains an open region isometric to a finite open cylinder.

\begin{prop} \label{metric}
There exist $\mathrm{O}(n)$-invariant Riemannian metrics $\tilde g$ on $S^n$ with $\Ric_{\tilde g} \geq 0$ and the following properties:
\begin{enumerate}
    \item[(i)] The reflection $(x_1,x_2\cdots,x_{n+1}) \mapsto (-x_1,x_2\cdots,x_{n+1})$ is an isometry of $(S^n,\tilde g)$.
    \item[(ii)] There exist $\delta \in (0,1)$ and $r>0$ such that $\{x \in S^n \mid |x_1|\leq\delta\}$ is isometric to a finite cylinder $[-r,r] \times S^{n-1}$. In particular, for any $t \in (-\delta,\delta)$, the spheres $\Sigma_t=\{x_1=t\}$ are degenerate stable minimal (in fact, totally geodesic) hypersurfaces in $(S^n,\tilde g)$.
\end{enumerate}
\end{prop}

\begin{proof}
Consider the following smooth, nonnegative, strictly decreasing function:
    \[\phi(x) = \left\{ \begin{array}{cl} e^{-\frac{x^2}{a^2-x^2}}, & \text{for} \ 0\leq x<a,\\  0, &\text{for}\ x \geq a \end{array}\right.,\]
where $a\in(1,2)$ is the unique number such that $\int_0^a e^{-\frac{x^2}{a^2-x^2}} \,dx=1$. We will use in our construction the fact that $0\leq \phi \leq 1$, and all derivatives of odd order of $\phi$ vanish at $t=0$.

Let $\rho$ be a primitive of $\phi$ with $\rho(0)=0$, that is, the nonnegative smooth function defined for $x \geq 0$ by
    \[\rho(x) ={\textstyle \int_0^x \phi(t)\,dt},\]\medskip
and consider, for $R>a$, the warped product metric $\tilde g=dt^2+\rho^2(t)ds_{n-1}^2$ on $(0,R]\times S^{n-1}$. Here $ds^{n-1}$ denotes the round metric of curvature $1$ on $S^{n-1}$. We claim that this induces a metric on a hemisphere in $S^{n}$ which is isometric to a cylinder near the boundary, and hence can be reflected to obtain a metric on $S^n$ with the required properties.

First, note that $\rho(0) = 0$ and $\rho' = \phi$, hence $\rho'(0) = 1$, and all even-order derivatives of $\rho$ vanish at $t=0$. Therefore, see e.g. \cite[Section 1.4.4]{petersen2006riemannian}, it defines a smooth metric on the hemisphere (by identifying $0\times S^{n-1} \subset [0,R]\times S^{n-1}$ to a single point). Moreover, we have  $\rho(x) \equiv 1$ for $x \geq a$, and hence $\tilde g=dt^2 + ds_{n-1}^2$ on $[a,R]\times S^{n-1}$, which is the product metric on the cylinder $[a,R] \times S^{n-1}$. In particular, the boundary $\{R\} \times S^{n-1}$ is totally geodesic with respect to the metric $\tilde g$. Therefore, one can use the reflection with respect to this boundary to extend the metric smoothly to the double of the hemisphere, which is diffeomorphic to $S^n$.

Second, by a direct computation (see e.g. \cite[Section 4.2.3]{petersen2006riemannian}), we have the following formula of Ricci curvature:
\begin{equation*}
\begin{aligned}
&\text{Ric} (\partial_t) = - (n-1) \frac{\rho''}{\rho} \partial_t, \\
&\text{Ric} (X) = \left( (n - 2) \frac{1-(\rho')^2}{\rho^2} - \frac{\rho''}{\rho}\right) X, \quad \text{if} \ X \  \text{is tangent to} \ S^{n-1}.
\end{aligned}
\end{equation*}

Since $\rho'' = \phi' \leq 0$ and $\rho' = \phi \in [0,1]$, we have $-\frac{\rho''}{\rho} \geq 0$ and $\frac{1-(\rho')^2}{\rho^2} \geq 0$. Therefore, the Ricci curvature of the metric $\tilde g$ is nonnegative.
\end{proof}

\begin{rmk} \label{positivity}
By the construction of the warping function $\rho$, we note that $\rho'(t)=1$ only at $t=0$, and $\rho''(t)=0$ only at $t=0$ and for $t\geq a$. Moreover, a direct computation shows that:
    \[\lim_{t \to 0^+} \frac{1 -\rho'(t)^2}{\rho(t)^2} = \lim_{t \to 0^+} \frac{-\rho''(t)}{\rho(t)} = \lim_{t \to 0^+}\frac{-\phi''(t)}{\phi(t)} = \frac{2}{a^2}.\]
Therefore, $\Ric_{\tilde g}(X,X)=0$ for a tangent vector $X \in T_pS^n$ if and only if $p$ lies in the domain $\{x \in S^n \mid |x_1|\leq \delta\}$ and $X = {\partial_t}|_p$. 

In addition, the scalar curvature is given by $-2(n-1) \frac{\rho''}{\rho} + (n-1)(n-2) \frac{1-(\rho')^2}{\rho^2}$, so the metric $\tilde g$ has positive scalar curvature.
\end{rmk}

\begin{rmk} \label{comparehalfsphere}
Numerical computation shows that the constant $a \simeq 1.65714 > \frac{\pi}{2}$. We observe that the two caps $\{x \in S^n \mid \pm x_1 > \delta\}$ are slightly larger than half-spheres by comparing the warped metric constructed above and the round metric. In order to prove this observation, we only need to show that the graph of $\rho(t)$ over $[0,a]$ is on top of the graph of $\sin (x - a + \frac{\pi}{2})$ over $[a - \frac{\pi}{2}, a]$. Let $f(b) = \int_0^{a-\frac{\pi}{2} + b} e^{-\frac{x^2}{a^2-x^2}} dx - \sin b$, $b \in [0, \frac{\pi}{2}]$, we want to show that $f(b) \geq 0$. Since $f(\frac{\pi}{2}) = 0$, it is suffice to prove that $f'(b) \leq 0$. 
By basic calculus, $\frac{2a^2 (a-\frac{\pi}{2} + b)}{(a^2 - (a-\frac{\pi}{2} + b)^2)^2} \geq \tan b$ for $b \in [0,\frac{\pi}{2}]$. Thus the derivative of $e^{\frac{(a-\frac{\pi}{2} + b)^2}{a^2-(a-\frac{\pi}{2}+b)^2}} \cos b$ is nonnegative, and its value at $b = 0$ is greater than $1$, therefore $e^{\frac{(a-\frac{\pi}{2} + b)^2}{a^2-(a-\frac{\pi}{2}+b)^2}} \cos b > 1$. Hence $f'(b) = e^{-\frac{(a-\frac{\pi}{2} + b)^2}{a^2-(a-\frac{\pi}{2}+b)^2}} - \cos b \leq 0$.

\end{rmk}

\medskip

\begin{proof}[Proof of \ref{thma}]

By arguing as in \cite{caju2022ground}, and observing that the middle sphere $\{0\} \times S^1$ divides $(S^n,\tilde g)$ in two isometric domains, we can construct, for sufficiently small $\e>0$, a solution $v_\e$ of \eqref{AC} such that $\{v_\e = 0\} = \Sigma_0$. In fact, one minimizes the Allen-Cahn energy on one of the domains bounded by $\Sigma_0$, among $W^{1,2}$ functions with vanishing trace, and extends this to a solution on $(S^n,\tilde g)$ by an odd reflection with respect to $\Sigma_0$.\medskip

We will show that this is a unique minimal sphere among $\Sigma_t$, for $|t|\leq \delta$, which can be obtained as the limit interface associated to solutions of \eqref{AC}. To fix some notations, for every $t \in [-1,1]$, denote $D^t = \{x\in S^n \mid x_1<t\}$ and $D_{t} = \{x \in S^n \mid x_1>t\}$. We let $\tau$ be the reflection of $S^n$ with respect to the hyperplane $\{x_1=0\}$, i.e. $$\tau(x_1, x_2, \cdots, x_{n+1}) = (-x_1, x_2, \cdots, x_{n+1}).$$
Note that $\tau$ is an involution and an isometry of $(S^n,\tilde g)$, by Proposition \ref{metric}, and that $\tau(D_t) = D_{-t}$ and $\tau(\overline{D_t}) = \overline{D_{(-t)}}$.

Suppose, by contradiction, that there exists an AC approximation $u_{\epsilon_j}$ of a sphere $\Sigma_t$ with $t \in (0,\delta)$. By \cite[Theorem 1]{HT}, and by flipping the sign of $u_{\epsilon_j}$ if necessary, we may assume that $u_{\epsilon_j}$ converges uniformly to $+1$ in any compact subset of $D^t$. In particular, for sufficiently large $j$, we have $u_{\epsilon_j}>0$ in $\overline{D^{t/2}}$ . Denote by $\Omega^1_j$ the nodal domain of $u_{\epsilon_j}$ that contains $\overline{D^{t/2}}$ and let $\Omega_j^2 = S^n \setminus \overline{\Omega_j^1}$, so that $\overline{\Omega_j^2} \subset S^n\setminus \overline{D^{t/2}}=D_{t/2}$ for large $j$, and $u_{\epsilon_j}=0$ on $\partial \Omega_j^2$. We also note that $\partial \Omega_j^2 \subset D_{t/2}$ and $S^n \setminus \overline{\Omega_j^2} \subset \overline{\Omega_j^1}$.

\begin{figure}[h]
\centering
\includegraphics[width=.75\textwidth]{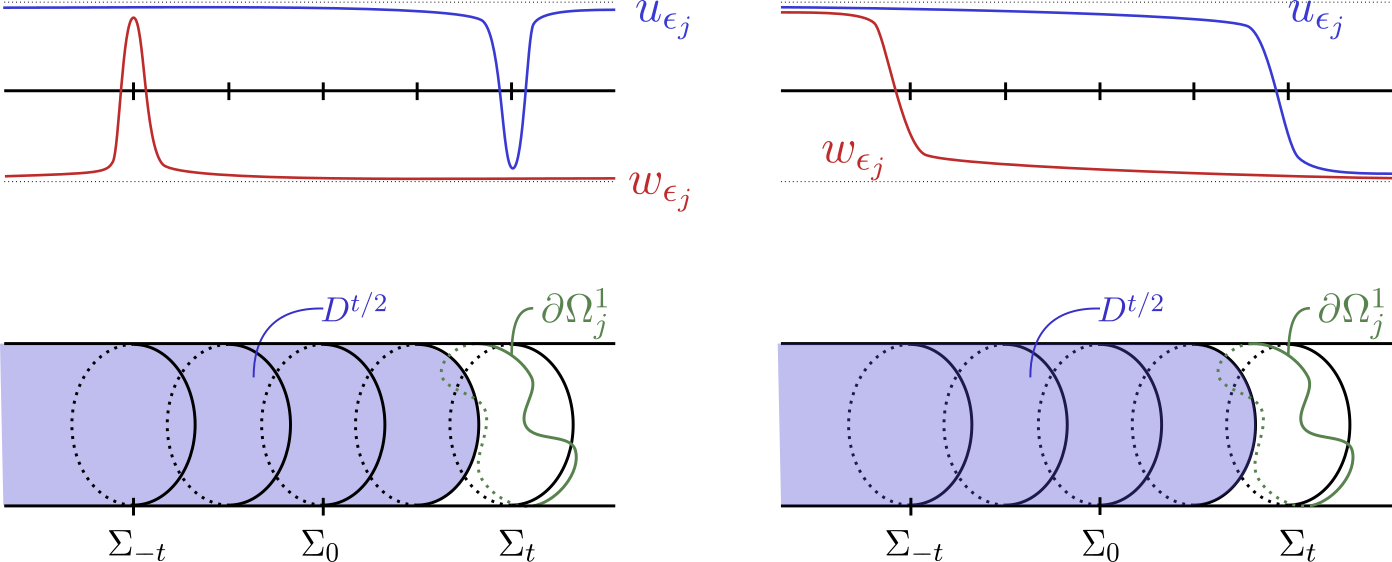}
\caption{Possible profiles of AC approximations of $\Sigma_t$ and nodal domains}
\label{pic_comparison}
\end{figure}

We aim to show that, regardless of the multiplicity that $\Sigma_t$ is obtained as a limit interface, we can use the symmetries of $(S^n,\tilde g)$ to construct another nonconstant solution $w_{\epsilon_j}$ such that $u_{\e_j} \geq w_{\e_j}$, contradicting the \ref{nonnegative Ricci FP}. Some examples of possible profiles of such solutions, with even or odd multiplicity, are pictured in Figure \ref{pic_comparison} above.\smallskip

Note that $w_{\epsilon_j}=-u_{\epsilon_j}\circ \tau$ is a solution of \eqref{AC} whose nodal set accumulates around $\tau(\Sigma_t)=\Sigma_{-t}$. Moreover, $\tau(\Omega_j^1) \supset \tau (\overline{D^{t/2}}) = \overline{D_{-t/2}}$, and $\tau(\Omega_j^2) \subset D^{-t/2}$ for large $j$. In addition, we have $w_{\epsilon_j}<0$ on the domain $\tau(\Omega_j^1)$ for large $j$, and also $S^n \setminus \overline{\tau(\Omega_j^2)} \subset \overline{\tau(\Omega_j^1)}$. 

It suffices to show that $u_{\e_j}\geq w_{\e_j}$, for sufficiently large $j$, on each of the following domains:
    \[\text{(i)} \ \overline{\Omega_j^2} \qquad  \qquad \text{(ii)} \ \overline{\tau(\Omega_j^2)}, \qquad\qquad \text{and} \qquad \qquad \text{(iii)} \ (S^n \setminus \overline{\Omega_j^2}) \cap (S^n \setminus \tau(\overline{\Omega_j^2})).\]
Then the Frankel Property for solution of \eqref{AC} implies that one of these solutions is constant, leading to a contradiction.\smallskip

\noindent (i) Since $\tau(\overline{\Omega_j^2}) \subset \tau(D_{t/2}) = D^{-t/2} \subset D^{t/2} \subset \Omega_j^1$, we know $-w_{\e_j}(x) = u_{\e_j}(\tau(x)) > 0$ for all $x \in \overline{\Omega_j^2}$, and $-u_{\e_j}(x) = 0$ for all $x \in \partial \Omega_j^2$. We conclude by Lemma \ref{comparison} that $-w_{\epsilon_j}> -u_{\epsilon_j}$ on $\overline{\Omega_j^2}$. Equivalently, $u_{\epsilon_j}>w_{\epsilon_j}$ on $\overline{\Omega_j^2}$.\smallskip

\noindent (ii) Using the fact that $\tau \circ \tau = \mathrm{id}_{S^n}$, we have $u_{\e_j}(x) = - w_{\e_j}(\tau(x)) > - u_{\e_j}(\tau(x)) = w_{\e_j}(x)$ for all $x \in \tau(\overline{\Omega_j^2})$.\smallskip

\noindent (iii) Note that $(S^n \setminus \overline{\Omega_j^2}) \cap (S^n \setminus \tau(\overline{\Omega_j^2})) \subset \overline{\Omega_j^1} \cap \overline{\tau(\Omega_j^1)}$. We know $u_{\e_j} \geq 0$ on $\overline{\Omega_j^1}$ and $w_{\e_j} \leq 0$ on $\overline{\tau(\Omega_j^1)}$. Hence $u_{\epsilon_j}\geq 0 \geq w_{\epsilon_j}$ in $\Omega_j^1 \cap \tau(\Omega_j^1)$ for sufficiently large $j$.
\end{proof} \medskip

\begin{rmk} \label{general}
The above proof works for a more general situation. We prove that for any nonconstant solution to \eqref{AC}, its nodal domain cannot contain half of the sphere, i.e. $\{x_1 \leq 0\}$ or $\{x_1 \geq 0\}$. Therefore the limit interface of nonconstant solutions to \eqref{AC} cannot completely lie on one side of the middle sphere $\Sigma_0$.
\end{rmk}

Thus we get our first observation that there are connected degenerate minimal hypersurfaces which cannot be obtained as the limit interface of solutions to the Allen-Cahn equation.
\medskip

\section{Min-max theory and Proof of \ref{thmb}} \label{min-max}

In this section, we compute the $1$-width of $(S^n, \tilde g)$, where $\tilde g$ is a Riemannian metric with the properties described in \ref{thma}, and describe the space of min-max minimal hypersurfaces achieving this width for both Almgren-Pitts and Allen-Cahn min-max procedures in $(S^n,\tilde g)$.

\subsection{The Almgren-Pitts and Allen-Cahn widths of a closed Riemannian manifold} First, we recall some notations from geometric measure theory and min-max theory. We follow the exposition of \cite{MarquesNevesMultiplicity,Dey} and refer to these articles and the references therein for details. Let $(M^n,g)$ be a closed Riemannian manifold.

Denote by $\mathbf{I}_k(M;\Z_2)$ the space of $k$-dimensional mod $2$ flat chains in $M$. For $k=n$, this space can be identified with the space of Caccioppoli sets (that is, finite perimeter) in $M$. Let also $\mathcal{Z}_{n-1}(M;\Z_2)$ be the space of flat $(n-1)$-dimensional mod 2 chains $T \in \mathbf{I}_n(M;\Z_2)$ such that $T=\partial U$ for some $U \in \mathbf{I}_n(M;\Z_2)$. The latter space can be endowed with the $\mathbf{F}$ \emph{metric} given by
   \[\mathbf{F}(T,S) = \mathcal{F}(T,S)+{\displaystyle \sup_h(|T|(f)-|S|(f))},\]
where $\mathcal{F}$ is the flat metric (see \cite[Section 2]{MarquesNevesMultiplicity}), $|T|$ and $|S|$ are the integral varifolds associated to $T$ and $S$, and the supremum is taken over real-valued Lipschitz functions $f$ defined on the Grassmannian space $G_{n-1}(M)$ with $|f|\leq 1$ and $\mathrm{Lip}(f)\leq 1$.

In \cite{Almgren}, Almgren computed the homotopy groups of spaces of flat cycles in closed Riemannian manifolds. These results imply, in particular, that $ \mathbf{I}_{n}(M;\Z_2)$ is simply connected, $\pi_1(\mathcal{Z}_{n-1}(M;\Z_2))\simeq \Z_2$, and the boundary map $\partial \colon \mathbf{I}_n(M;\Z_2)\to \mathcal{Z}_{n-1}(M;\Z_2)$ is a $2$-cover. Let $X$ be a cubical complex which has a non-trivial $2$-cover $\pi\colon \tilde X \to X$. The class of continuous maps $\Phi\colon X \to \mathcal{Z}_{n-1}(M;\Z_2)$ which are continuous with respect to the $\mathbf{F}$ metric and such that the induced maps between the fundamental groups satisfy $\ker \Phi_* = \mathrm{im} \pi_*$ is a well-defined homotopy class \cite[Proposition 2.5]{Dey}. Its associated \emph{Almgren-Pitts width} is the number
    \[\mathbf{L}_{AP}(\Pi;M,g)=\mathbf{L}_{AP}(\Pi):=\adjustlimits \inf_{\Phi \in \Pi}\sup_{x \in X} \mathbf{M}(\Phi(x)),\]
where $\mathbf{M}(T)$ is the \emph{mass} (or, intuitively, the area) of a cycle $T \in \mathcal{Z}_{n-1}(M;\Z_2)$ calculated with respect to the metric $g$.

We denote by $\mathbf{C}_{AP}(\Pi)$ the set of min-max stationary varifolds with mass $\mathbf{L}_{AP}(\Pi)$. Concretely, we say that a stationary integral $(n-1)$-varifold $V$ with optimal regularity (in the sense of \cite{Pitts,SchoenSimon}) is in $\mathbf{C}_{AP}(\Pi)$ if its mass is $\mathbf{L}_{AP}(\Pi)$ and if there exists sequences $\{\Phi_i\} \subset \Pi$ and $\{x_i\} \subset X$ such that $\mathbf{M}(\Phi_i(x_i))=\sup \mathbf{M}(\Phi_i)\to \mathbf{L}_{AP}(\Pi)$ and $|\Phi_i(x_i)| \to V$ in the weak varifold topology.

In \cite{Dey}, Dey compared the Almgren-Pitts width $\mathbf{L}_{AP}(\Pi)$ and the critical set $\mathbf{C}_{AP}(\Pi)$ with the limit interface associated to min-max solutions of Allen-Cahn in the sense of \cite{GasparGuaraco,Guaracominmax}. Concretely, let $\tilde \Pi$ be the space of continuous $\Z_2$-equivariant maps $h\colon \tilde X \to W^{1,2}(M)\setminus \{0\}$, where the $\Z_2$ actions in $\tilde X$ and $W^{1,2}(M)\setminus\{0\}$ are the deck transformation associated to $\pi\colon \tilde X \to X$ and the antipodal map $u \mapsto -u$ in $W^{1,2}(M)$, respectively. The $\e$-\emph{Allen-Cahn width} associated to $\tilde \Pi$ is defined by
    \[\mathbf{L}_\e(\tilde \Pi;M,g)=\mathbf{L}_{\e}(\tilde \Pi) := \adjustlimits \inf_{h \in \tilde \Pi}\sup_{x \in \tilde X} E_\e(h(x)).\]
By \cite[Theorem1.2]{Dey}, the Allen-Cahn width of $\Pi$ can be obtained as the limit of the $\e$-Allen-Cahn widths associated to $\tilde \Pi$, that is
    \begin{equation}\label{AP-AC widths}
        \frac{1}{2\sigma}\lim_{\epsilon \to 0}\mathbf{L}_\e(\tilde \Pi)= \mathbf{L}_{AP}(\Pi)
    \end{equation}
    
Finally, we denote by $\mathbf{C}_{AC}(\tilde \Pi)$ the set of all stationary integral $(n-1)$-varifolds $V$ in $M$ with optimal regularity, in the sense of \cite{TW,Guaracominmax} (see also \cite[Theorem 2.7]{Dey}), and such that $V$ is the varifold limit of a sequence $V_i$ of varifolds associated to min-max solutions $u_i$ of \eqref{AC}, with $\e=\e_i \downarrow 0$, corresponding to $\tilde \Pi$. This means that there exists a sequence $\{h_i\}$ in $\tilde \Pi$ such that $\sup_{x \in \tilde X}E_{\e_i}(h_i(x)) \to \mathbf{L}_\e(\tilde \Pi)$ and $\mathrm{dist}_{W^{1,2}}(u_i,h_i(\tilde X))\to 0$.

In \cite[Theorems 1.4]{Dey}, Dey proved that $\mathbf{C}_{AC}(\tilde \Pi)\subset \mathbf{C}_{AP}(\Pi)$. Our \ref{thmb} can be stated as follows: if $\tilde g$ is a metric on $S^n$ as given by Proposition \ref{metric}, then $\mathbf{L}_{AC}(\Pi)$ is the first min-max width $\omega_1(S^n,\tilde g)$, in the sense of \cite{MarquesNevesRicci}, and $\mathbf{C}_{AC}(\tilde \Pi)=\{|\Sigma_0|\}$ is strictly contained in $\mathbf{C}_{AP}(\Pi)$.

\subsection{The \texorpdfstring{$1-$}{1-}width of \texorpdfstring{$(S^n,\tilde g)$}{(S3,g~)} and its critical set} \label{1width}

We describe the first min-max width of the Riemannian manifolds $(S^n,\tilde g)$ in Theorem \ref{thma} assuming the length of the cylindrical region is sufficiently large. This leads to the following observation: even though the Almgren-Pitts $1$-width and the Allen-Cahn $1$-width are the same (see \cite{Dey}), the sets of minimal hypersurfaces that realize these widths are different.

We consider the canonical sweepout $\Phi\colon [-1,1] \mapsto \mathcal{Z}_{n-1}(M;\Z_2)$ given by $\Phi(t) = \partial \llbracket \{x_1 <t\}\rrbracket = \llbracket \Sigma_t \rrbracket$, namely the family of mod $2$ flat cycle associated to each sphere $\Sigma_t$. It induces a map from $S^1$ to $\mathcal{Z}_{n-1}(M,\Z_2)$ which we still denote by $\Phi$ which belongs to the homotopy class $\Pi$ described above (with associated double cover $e^{i\theta} \in S^1 \mapsto e^{i2\theta} \in S^1$). Since
    \[\max_{t \in [-1,1]}\mathbf{M}(\llbracket \Sigma_t \rrbracket) = \max_{t \in [0,R]} \rho(t)^{n-1}\cdot \area(S^{n-1}) = \area(S^{n-1}) = \mathbf{M}(\llbracket \Sigma_0 \rrbracket),\]
where the area of $S^{n-1}$ is computed with respect to the round metric, we conclude
    \[\omega_1(S^n,\tilde g) \leq \mathbf{L}_{AP}(\Pi; S^n,\tilde g) \leq \area(S^{n-1}).\]
We will now show that the equality holds, that is:

\begin{prop} \label{width}
There exists $R=R(n)>0$ with the following property. If the Riemannian metric $\tilde g$ is so that $r>R(n)$, then the $1-$width of $(S^n, \tilde g)$ is area of the minimal spheres $\Sigma_t$, for $|t|\leq\delta$.
\end{prop}

\begin{proof}
Let $A = \{x \in S^n \mid |x_1| \leq \delta\}$ be the cylindrical region in $(S^n,\tilde g)$. It follows from \cite{LMNWeyl}, Proposition 2.15 that $\omega_1(S^n,\tilde g) \geq \omega_1(A,\tilde g) = \omega_1([-r,r]\times S^{n-1})$, where the cylinder has the product metric. Moreover, it follows from the proof of \cite[Theorem 9]{Song} that there exists $\alpha=\alpha(n)>0$ such that the 1-width of $[r,r]\times S^{n-1}$ is the area of a totally geodesic slice $t \times S^{n-1}$ provided $r>\alpha$. Therefore, $\omega_1(S^n,\tilde g)\geq \area(S^{n-1})$ for such metrics. By the observation above, we conclude $\omega_1(S^n,\tilde g)=\area(S^{n-1})$.
\end{proof}

We also recall the following characterization of the least area minimal hypersurfaces in $(S^n,\tilde g)$:

\begin{prop} \label{least area}
We can choose the metrics $\tilde g$ so that any smooth, embedded, closed minimal hypersurface in $(S^n,\tilde g)$ that has area $\leq \area(\Sigma_0)$ is a sphere $\Sigma_t$, for some $|t|\leq\delta$.
\end{prop}

\begin{proof}
It follows from the maximum principle that any connected, closed minimal hypersurface that is contained in the cylindrical region $\{x \in S^{n} \mid |x_1|< \delta\} \simeq (-r,r)\times S^{n-1}$ is a minimal sphere $\Sigma_t$. On the other hand, we can argue once again as in the proof of \cite[Theorem 9]{Song} to see that, by the monotonicity formula, we can choose $r$ (which, using the notation of the proof of Proposition \ref{metric}, corresponds to the length $R-a$ of the interval where $\rho$ is constant) so that any minimal hypersurface that is not contained in that region has area $>\area(\Sigma_0)$.
\end{proof}

\begin{rmk}
We observe that any closed, stable minimal hypersurface $\Gamma$ in $(S^n,\tilde g)$ must be one of the spheres $\Sigma_t$ with $|t|<\delta$. In fact, since $\Ric_{\tilde g}\geq 0$, by the second variation formula (see e.g. \cite[Corollary 1.33]{colding2011course}), we have $\Ric_{\tilde g}(N,N)\equiv 0$ along $\Gamma$, where $N$ is a unit normal vector field. By Remark \ref{positivity}, we see that $\Gamma$ is contained in the cylindrical region $\{x \in S^{n} \mid |x_1|\leq \delta\}$, and $N = {\partial_t}$ along $\Gamma$. Therefore $\Gamma$ is a totally geodesic sphere orthogonal to $\partial_t$.
\end{rmk}

\begin{rmk}
Even though the argument presented here does not give an explicit estimate for the smallest $R(n)>0$ for which our results hold true, we observe that the next proof remains valid as long as $\omega_1(S^n,\tilde g)=|\Sigma_0|$. As observed in the introduction, it seems reasonable to expect that this is true for any $r>0$.
\end{rmk} \medskip

\begin{proof}[Proof of \ref{thmb}]

By Proposition \ref{width}, we see that
    \[\omega_1(S^n,\tilde g) \leq \mathbf{L}_{AP}(\Pi; S^n,\tilde g) \leq \mathbf{M}(\llbracket \Sigma_0 \rrbracket) = \omega_1(S^n,\tilde g),\]
In particular, this proves the first property stated in \ref{thmb}. It also shows that all the minimal spheres $\{\Sigma_t\}_{t \in [-\delta,\delta]}$ achieve the first width of $\omega_1(S^n,\tilde g)$.

In \cite{Dey}, Dey proved that the Almgren-Pitts $1$-width and the Allen-Cahn $1$-width are the same, namely:
    \[\mathbf{L}_{AP}(\Pi; S^n,\tilde g) = {\textstyle \frac{1}{2\sigma}}\lim_{\e \to 0^+}\mathbf{L}_\e(\tilde \Pi).\]
In particular, if $V \in \mathbf{C}_{AC}(\tilde \Pi)$, then $\|V\|(M) = \omega_1(S^n,\tilde g) = \area(\Sigma_0)$. 
But then Proposition \ref{least area} implies that $V$ is the varifold induced by one of the spheres $\Sigma_t$ for some $|t|\leq \delta$. By \ref{thma}, we conclude that $V$ is induced by $\Sigma_0$ and $\mathbf{C}(\tilde\Pi)=\{|\Sigma_0|\}$, namely and the Allen-Cahn $1$-width of $(S^g,\tilde g)$ is realized only by the middle sphere $\Sigma_0$.
\end{proof}

\begin{proof}[Proof of \ref{corollary}]
For each $\epsilon>0$, let $u_{\epsilon}$ be a least energy nonconstant solution of \eqref{AC} in $(S^n,\tilde g)$. We recall that by \cite[Theorem 2.3]{caju2019solutions}, such solution is an index 1 min-max solution associated to the first Allen-Cahn width, that is, such that $E_\e(u_\e) = \mathbf{L}_\e(\tilde \Pi)$. Therefore, the only limit interface associated to $\{u_\e\}$ is $\Sigma_0$. By Remark \ref{index} in the introduction, it follows that $u_{\epsilon}$ has nullity $0$, for sufficiently small $\epsilon>0$.

Consider now a $1$-parameter family $R_t$ of isometries of $S^{n-1}$ (endowed with the round metric) generated by a Killing field $X$. Since $\tilde g$ was constructed as a warped product metric on a cylinder, $R_t$ naturally gives rise to a $1$-parameter family of isometries of $(S^n, \tilde g)$ fixing the poles where $x_1 = \pm 1$. The associated Killing field then extends to a Killing field $X$ of $(S^n,\tilde g)$ which is parallel to the spheres $\{x_1=c\}$, for $c \in [-1,1]$ and which vanishes at the poles.

It follows that $u_{\epsilon}\circ R_t$ is a $1$-parameter family of solutions of (AC) in $(S^n, \tilde g)$, and hence $\partial_t(u_\epsilon \circ R_t) =\langle \nabla u_\epsilon, X \rangle$ belongs to the nullity of $u_{\epsilon}$. This proves that $\langle \nabla u_{\epsilon}, X \rangle =0$ for any such family of isometries. Since the corresponding Killing fields span any tangent space to the spheres $\{x_1=c\}$, it follows that $u_\epsilon$ is constant along them, namely, that $u_\epsilon$ is rotationally symmetric.

Since $u_{\epsilon}$ has $\Sigma_0$ with \emph{multiplicity one} as its associated limit interface, it follows that its nodal set $\{u_{\epsilon}=0\}$ is connected and hence it can be written as $\{x_1=a_\epsilon\}$ for some $a_\epsilon \downarrow 0$. Then we can argue as in the proof of \ref{thma}, or directly use the Frankel property \cite[Proposition 10]{H} to conclude that $\{u_\epsilon =0\} = \{x_1=0\} =\Sigma_0$. By the uniqueness of the Dirichlet solution on the domain $\{x_1 \leq 0\}$, e.g. \cite{BO}, we conclude that $u_\epsilon = \pm v_\epsilon$.
\end{proof}

\begin{rmk} \label{nonsymmetric}
Consider a metric $\overline{g}$ on $S^n$ with nonnegative Ricci curvature which satisfies property (ii) in \ref{thma}. The comparison argument explored here indicates that, for each positive integer $m$, there should be at most one minimal slice $\Sigma_t =\{x_1 =t\}$ which can be obtained as a limit interface with multiplicity $m$. Nevertheless, if $\overline g$ is not invariant by a reflection, we cannot proceed exactly as in the proof of \ref{thma} to reach a contradiction (nor use the same strategy to construct an Allen-Cahn solution that vanishes precisely along $\Sigma_0$). In addition, the comparison argument, does not directly rule out approximations of different slices arising form different sequences of parameters $\e_j \downarrow 0$. 

Nevertheless, for $m=1$, we believe that it should be possible to use a Lyapunov-Schmidt reduction argument to produce Allen-Cahn solutions $v_\e$ with a single layer near \emph{some} minimal slice $\Sigma_{t_0}$ \emph{for all sufficiently small} $\e>0$, ruling out any other minimal slice as a multiplicity one limit interface. Furthermore, if the length of the cylindrical region is sufficiently large, then $\Sigma_{t_0}$ is the least area minimal hypersurface in $(S^n,\overline g)$, and $v_\e$ should be the only least energy unstable critical point of the energy, which also achieves the first Allen-Cahn width of $(S^n,\overline g)$. As mentioned in the introduction, a closely related construction for Neumann solutions in planar domains which contain a rectangle $[0,1]\times [0,b]$ can be found in \cite{KowalczykThesis,AFK}. In these works, the authors construct solutions which approximate the degenerate line segment $\{1/2\}\times [0,b]$. In view of these constructions, we expect that the only multiplicity one limit interface should be located near the middle of the (maximal) cylindrical region in $(S^n,\overline g)$.
\end{rmk}

\bibliography{main}
\bibliographystyle{acm}

\end{document}